\documentclass[11pt]{article}

\usepackage[utf8]{inputenc}
\usepackage[T1]{fontenc}
\usepackage[english]{babel}

\usepackage{amsmath,amssymb,amsfonts,latexsym}
\usepackage{mathtools}
\usepackage{bm}

\usepackage{amsthm}
\numberwithin{equation}{section}

\usepackage{graphicx}
\usepackage{subcaption}
\usepackage{float}

\usepackage{enumitem}
\usepackage{microtype}

\usepackage{xcolor}
\usepackage[colorlinks=true,linkcolor=blue,citecolor=blue,urlcolor=blue]{hyperref}

\usepackage{tikz}

\allowdisplaybreaks

\theoremstyle{plain}
\newtheorem{theorem}{Theorem}[section]
\newtheorem{lemma}[theorem]{Lemma}
\newtheorem{proposition}[theorem]{Proposition}
\newtheorem{corollary}[theorem]{Corollary}

\theoremstyle{definition}
\newtheorem{definition}[theorem]{Definition}

\theoremstyle{remark}
\newtheorem{remark}[theorem]{Remark}

\title{\Large
${}^{W}D_{t}^{\alpha,\beta}$: A Volterra Fractional Time Operator\\
with Sharp Bernstein Threshold and Regularized Memory
}

\author{
Mohamed Wakrim\thanks{Corresponding author.}\\
Ibn Zohr University, FSA, Morocco\\
\texttt{m.wakrim@uiz.ac.ma}
}

\date{\today}

\begin{document}
\maketitle

\begin{center}
\small
\textbf{Preprint.} This manuscript is a preprint and may be under journal review.
\end{center}

\begin{abstract}
We introduce a new two-parameter fractional time operator with Volterra structure,
denoted by ${}^{W}D_{t}^{\alpha,\beta}$, defined through the Laplace symbol
\[
\Phi_{\alpha,\beta}(s)
=
\frac{s^\alpha}{\bigl(1+(1-\alpha)s^{\alpha-1}\bigr)^\beta},
\qquad 0<\alpha<1,
\ \beta\ge0.
\]
The operator preserves the Caputo-type high-frequency behavior while allowing a
controlled modification of the low-frequency regime via $\beta$. We develop an
explicit symbolic/Volterra theory: Prabhakar-type kernels, a left-inverse Volterra
integral, and a fractional fundamental theorem of calculus.
A central contribution is a sharp clarification of the Bernstein structure of the
symbol. We show that the natural factorization
$\Phi_{\alpha,\beta}(s)=s^\alpha h_\alpha(s)^\beta$ does not fit the classical
Bernstein product mechanism for any $\beta>0$. Nevertheless, by a direct
complete-monotonicity argument on $\Phi'_{\alpha,\beta}$, we prove the exact
Bernstein threshold
\[
\Phi_{\alpha,\beta}\in\mathcal{BF}
\quad\Longleftrightarrow\quad
0\le\beta\le1.
\]
where $\mathcal{BF}$ denotes the class of Bernstein functions

\noindent For $\beta>1$, the Bernstein property fails by a low-frequency asymptotic
convexity obstruction. This shows that the Bernstein nature of the natural range
$0\le\beta\le1$ is genuine but is not produced by the standard product mechanism.
We then establish well-posedness of abstract W-fractional Cauchy problems with
sectorial generators by resolvent estimates and Laplace inversion, yielding a
W-resolvent family with temporal regularity and smoothing properties. As an
illustration, we apply the theory to a W-fractional diffusion model and discuss
the effect of $\beta$ on the relaxation of spectral modes.
\end{abstract}


\section{Introduction}

Fractional evolution equations play a central role in the modeling of
diffusion processes with memory, viscoelastic phenomena, and anomalous transport; see, for instance, \cite{Podlubny1999,KilbasSrivastavaTrujillo2006,Mainardi2010,Pruss2016}. From an operator-theoretic viewpoint, the classical Caputo and
Riemann--Liouville derivatives are distinguished by their intimate
connection with Volterra convolution operators, Bernstein functions, and subordination theory, which together provide a powerful framework for well-posedness, regularity, and probabilistic interpretation \cite{SchillingSongVondracek2012,Bazhlekova2000}.

\medskip
\noindent In recent years, several fractional derivatives with modified or
regularized memory kernels have been proposed, most notably the
Atangana--Baleanu and Caputo--Fabrizio models \cite{CaputoFabrizio2015,AtanganaBaleanu2016}. These operators aim at removing certain singular features of classical fractional kernels, but often at the cost of losing the Volterra or
Bernstein structure that underpins the abstract theory of fractional
evolution equations. As a consequence, their interpretation within the standard subordination-based framework remains delicate and, in some cases, problematic \cite{Giusti2018AB,Luchko2019AB}.

\noindent The purpose of the present work is to introduce and analyze a new fractional time operator,
\[
{}^{W}D_{t}^{\alpha,\beta},
\]
designed to preserve the essential Volterra structure of classical
fractional derivatives while allowing a controlled regularization of
memory effects. The operator is defined through the Laplace symbol
\[
\Phi_{\alpha,\beta}(s)
=
\frac{s^\alpha}{\bigl(1+(1-\alpha)s^{\alpha-1}\bigr)^\beta},
\]
where the parameter $\beta$ modifies the low-frequency behavior without altering the Caputo-type high-frequency scaling. This design leads to a clear separation between short-time fractional
effects and long-time relaxation mechanisms.

\noindent A central structural result of the paper is that the natural auxiliary factor in the factorization $\Phi_{\alpha,\beta}(s)=s^\alpha h_\alpha(s)^\beta$ fails to be completely monotone for all $0<\alpha<1$, so that the canonical Bernstein product mechanism does not apply when $\beta>0$. Moreover, we show that the symbol $\Phi_{\alpha,\beta}$ is not a Bernstein function for $\beta>1$, placing the $W$-operator outside the classical subordination framework. Despite this, the operator retains a fully explicit Volterra structure and admits a well-posed abstract Cauchy problem based on resolvent methods.

\medskip
\noindent
After an earlier version of this manuscript, the author became aware of the
public staged proof APP-0019 of the Pudim AI Project, which pointed out the
complementary Bernstein range of the $W$-symbol for $0\le \beta\le 1$.
This observation led to the inclusion of the self-contained proof given in
Section~\ref{sec:positioning}. Combined with the non-Bernstein obstruction
for $\beta>1$ established here, this yields the sharp Bernstein threshold
\[
\Phi_{\alpha,\beta}^{W}\in\mathcal{BF}
\quad\Longleftrightarrow\quad
0\le \beta\le 1.
\]
We emphasize that this statement concerns the Bernstein property of the
symbol only. It does not assert the complete Bernstein property, nor does it
automatically restore the full complete-Bernstein/subordination framework.

\medskip
\noindent\textbf{On the role of the parameter $\beta$.}
Throughout the paper, the parameter $\beta$ plays a dual role.
On the one hand, the Volterra representation of the operator
${}^{W}D_{t}^{\alpha,\beta}$ and the reciprocal kernel construction remain valid for all $\beta\ge0$. On the other hand, positivity properties of the memory kernel, as well as the well-posedness theory for abstract evolution equations, naturally restrict the analysis to the range $0\le\beta\le1$. This interval corresponds to the regime in which Prabhakar-type kernels satisfy known positivity and regularity criteria and where the resolvent-based approach yields a robust evolution theory. The case $\beta>1$, while analytically meaningful at the symbolic level, raises additional questions that are briefly discussed but left for future investigation.

The paper is organized as follows.
In Section~\ref{sec:symbol} we introduce the symbol $\Phi_{\alpha,\beta}$ and analyze its asymptotic and sectorial properties. Section~\ref{sec:positioning} proves the sharp Bernstein threshold and clarifies the structural position of the $W$-operator with respect to Bernstein theory.
In Section~\ref{sec:volterra-inverse} we construct the Volterra
representation, the inverse operator, and establish a fractional
fundamental theorem of calculus. Section~\ref{sec:W-evolution} is devoted to abstract $W$-fractional evolution equations and their well-posedness. Finally, in Section~\ref{sec:model-application} we illustrate the theory on a $W$-fractional diffusion model and discuss the spectral and qualitative influence of the parameter $\beta$.

\section{The symbol $\Phi_{\alpha,\beta}$}
\label{sec:symbol}

\subsection{Motivation and design principles}

Fractional time--operators are naturally encoded through their Laplace symbols. In the classical Caputo case, the symbol $s^\alpha$
($0<\alpha<1$) is a complete Bernstein function, a fact that underlies the entire Volterra--Bernstein framework: complete monotonicity of the kernel, subordination, and the standard abstract Cauchy problem theory.

\noindent However, several non-singular models proposed in the literatureâ€”most
notably of Atangana--Baleanu typeâ€”are associated with symbols that saturate or flatten at high frequencies. From an operator--theoretic perspective, this saturation destroys the Caputo scaling, blurs the separation between short- and long-time regimes, and complicates resolvent representations by altering the sectorial geometry induced by $s^\alpha$.

\noindent The guiding principle here is therefore twofold:
\begin{itemize}
\item to \emph{preserve} the Caputo behaviour $\Phi(s)\sim s^\alpha$ at high frequencies (short-time regime);
\item to introduce a \emph{controlled low-frequency modification} allowing,
when desired, a transition toward a first-order regime.
\end{itemize}

This leads to the two-parameter family of symbols
\[
\Phi_{\alpha,\beta}(s)
=
\frac{s^\alpha}{\bigl(1+(1-\alpha)s^{\alpha-1}\bigr)^\beta},
\]
where $\beta\ge0$ tunes the low-frequency behaviour without affecting the high-frequency Caputo scaling. In the distinguished case $\beta=1$, the symbol interpolates between a fractional regime at infinity and a linear regime at the origin.

\subsection{Definition and basic properties}

Let $0<\alpha<1$ and $\beta\ge0$. We define
\begin{equation}\label{eq:Phi}
\Phi_{\alpha,\beta}(s)
=
\frac{s^\alpha}{\bigl(1+(1-\alpha)s^{\alpha-1}\bigr)^\beta},
\qquad s>0.
\end{equation}

\begin{lemma}[Regularity and asymptotics]\label{lem:basic}
The function $\Phi_{\alpha,\beta}$ is positive and $C^\infty$ on
$(0,\infty)$. Moreover,
\[
\Phi_{\alpha,\beta}(s)\sim s^\alpha \quad (s\to\infty),
\qquad
\Phi_{\alpha,1}(s)\sim \frac{s}{1-\alpha} \quad (s\to0^+),
\]
and for every fixed $s>0$,
\[
\lim_{\alpha\to1}\Phi_{\alpha,\beta}(s)=s.
\]
\end{lemma}

\begin{proof}
For $s>0$, all factors in \eqref{eq:Phi} are positive and smooth, and the denominator never vanishes, hence $\Phi_{\alpha,\beta}\in C^\infty(0,\infty)$.

As $s\to\infty$, since $\alpha-1<0$ one has $s^{\alpha-1}\to0$, so the denominator tends to $1$ and $\Phi_{\alpha,\beta}(s)\sim s^\alpha$.

\noindent For $\beta=1$ and $s\to0^+$,
\[
\Phi_{\alpha,1}(s)
=\frac{s}{s^{1-\alpha}+(1-\alpha)}
\sim \frac{s}{1-\alpha},
\]
because $1-\alpha>0$.

\noindent Finally, for fixed $s>0$, letting $\alpha\to1$ yields
$s^\alpha\to s$ and $(1-\alpha)s^{\alpha-1}\to0$, hence
$\Phi_{\alpha,\beta}(s)\to s$.
\end{proof}

\subsection{Sectorial behaviour and two-regime estimates}

The construction of resolvent families and mild solutions relies on
precise control of the symbol along sectorial contours. The symbol
$\Phi_{\alpha,\beta}$ exhibits a characteristic two-regime behaviour.

\begin{proposition}[Two-regime sectorial bounds]\label{prop:Phi_2regime}
Fix $0<\alpha<1$ and $\beta\ge0$. Then for every $\varepsilon\in(0,\pi)$ there
exist constants $c_\varepsilon,C_\varepsilon>0$ such that for all
$s\in\mathbb{C}$ with $|\arg s|\le \pi-\varepsilon$:

\begin{enumerate}[label=\textnormal{(\roman*)}]
\item if $|s|\ge 1$, then
\[
c_\varepsilon\,|s|^\alpha \le |\Phi_{\alpha,\beta}(s)| \le C_\varepsilon\,|s|^\alpha;
\]
\item if $0<|s|\le 1$, then
\[
c_\varepsilon\,|s|^{\alpha+\beta(1-\alpha)}
\le |\Phi_{\alpha,\beta}(s)|
\le C_\varepsilon\,|s|^{\alpha+\beta(1-\alpha)}.
\]
\end{enumerate}
Equivalently, uniformly on the sector $|\arg s|\le \pi-\varepsilon$,
\[
|\Phi_{\alpha,\beta}(s)| \simeq |s|^\alpha \quad (|s|\ge1),
\qquad
|\Phi_{\alpha,\beta}(s)| \simeq |s|^{\alpha+\beta(1-\alpha)} \quad (0<|s|\le1),
\]
where the implicit constants may depend on $\varepsilon$ (and on $\alpha,\beta$ if these are not fixed), but are uniform with respect to $s$ in the prescribed sector.
\end{proposition}

\begin{proof}
Let $s=re^{i\vartheta}$ with $r>0$ and $|\vartheta|\le\pi-\varepsilon$. We treat separately the regimes $r\ge1$ and $0<r\le1$.

\medskip\noindent
\textbf{Case 1: $r\ge1$.}
Since $0<\alpha<1$, we have $\alpha-1<0$ and therefore
\[
|s^{\alpha-1}| = r^{\alpha-1} \le 1.
\]
It follows that
\[
1 \le \bigl|1+(1-\alpha)s^{\alpha-1}\bigr|
\le 1+(1-\alpha)|s^{\alpha-1}|
\le 2-\alpha.
\]
Hence there exist constants $0<c_1\le C_1<\infty$, depending only on
$\alpha$, such that
\[
c_1 \le \bigl|1+(1-\alpha)s^{\alpha-1}\bigr| \le C_1
\qquad (r\ge1).
\]
Using the definition \eqref{eq:Phi}, we obtain
\[
|\Phi_{\alpha,\beta}(s)|
= \frac{|s|^\alpha}{\bigl|1+(1-\alpha)s^{\alpha-1}\bigr|^\beta}
\simeq r^\alpha,
\qquad r\ge1,
\]
with constants independent of $s$ in the sector $|\arg s|\le\pi-\varepsilon$.

\medskip\noindent
\textbf{Case 2: $0<r\le1$.}
In this regime,
\[
|s^{\alpha-1}| = r^{-(1-\alpha)} \ge 1.
\]
Moreover, since $|\arg s|\le\pi-\varepsilon$, the complex number
$s^{\alpha-1}$ remains uniformly away from the negative real axis, so
there exists $c_\varepsilon>0$ such that
\[
\bigl|1+(1-\alpha)s^{\alpha-1}\bigr|
\ge c_\varepsilon\,|s^{\alpha-1}|
= c_\varepsilon\,r^{-(1-\alpha)}.
\]
On the other hand,
\[
\bigl|1+(1-\alpha)s^{\alpha-1}\bigr|
\le 1+(1-\alpha)|s^{\alpha-1}|
\le C\,r^{-(1-\alpha)},
\]
for some constant $C>0$. Hence
\[
\bigl|1+(1-\alpha)s^{\alpha-1}\bigr|
\simeq r^{-(1-\alpha)},
\qquad 0<r\le1,
\]
uniformly for $|\arg s|\le\pi-\varepsilon$.

\noindent Using again \eqref{eq:Phi}, we deduce
\[
|\Phi_{\alpha,\beta}(s)|
= \frac{r^\alpha}{\bigl|1+(1-\alpha)s^{\alpha-1}\bigr|^\beta}
\simeq r^{\alpha+\beta(1-\alpha)},
\qquad 0<r\le1,
\]
with constants depending at most on $\varepsilon$, $\alpha$, and $\beta$, but uniform with respect to $s$ in the prescribed sector.
\end{proof}

\medskip
\noindent The estimates in Proposition~\ref{prop:Phi_2regime} make explicit the frequency structure encoded by $\Phi_{\alpha,\beta}$: Caputo-type behaviour at high frequencies and a tunable low-frequency regime. This structural feature will be central both in resolvent representations and in the sharp Bernstein threshold discussed next.

\section{Sharp Bernstein threshold and structural positioning}
\label{sec:positioning}

This section clarifies the position of the $W$-operator
${}^{W}D_{t}^{\alpha,\beta}$ with respect to the classical
Bernstein--subordination framework. We distinguish two different issues.
First, the natural factorization of the symbol does \emph{not} activate the
standard Bernstein product mechanism when $\beta>0$. Second, the symbol itself
has a sharp Bernstein threshold: it is a Bernstein function exactly for
$0\le\beta\le1$, and it is not Bernstein for $\beta>1$.

Throughout this section, we denote by $\mathcal{CM}$ the class of completely
monotone functions on $(0,\infty)$, namely
\[
\mathcal{CM}
:=
\{
g\in C^\infty(0,\infty),:,
(-1)^n g^{(n)}(s)\ge0,\quad n\ge0,\ s>0
\}.
\]
We denote by $\mathcal{BF}$ the class of Bernstein functions on $(0,\infty)$,
that is,
\[
\mathcal{BF}
:=
\{
f\in C^\infty(0,\infty),:,
f(s)\ge0\ \text{and}\ f'\in\mathcal{CM}
\}.
\]
Equivalently,
\[
f\in\mathcal{BF}
\quad\Longleftrightarrow\quad
f(s)\ge0
\ \text{and}
(-1)^n f^{(n+1)}(s)\ge0,
\quad n\ge0,\ s>0.
\]

\subsection{Failure of the canonical Bernstein product mechanism}

A common way to identify Bernstein symbols is to use closure properties
involving complete Bernstein functions and completely monotone or
Stieltjes-type factors; see \cite{SchillingSongVondracek2012}. In our setting,
a natural attempt is to factorize
\[
\Phi_{\alpha,\beta}(s)
= s^\alpha h_\alpha(s)^\beta,
\qquad
h_\alpha(s)=\frac{1}{1+(1-\alpha)s^{\alpha-1}},
\]
where $s^\alpha$ is a complete Bernstein function for $0<\alpha<1$.
Thus, any direct product-type Bernstein argument based on this factorization
would require suitable monotonicity properties of the auxiliary factor
$h_\alpha$ or of its powers. The following result shows that even the basic
complete monotonicity property fails.

\begin{proposition}[Non-complete monotonicity of $h_\alpha$]
\label{prop:h_not_cm}
For every $0<\alpha<1$, the function
\[
h_\alpha(s)=\frac{1}{1+(1-\alpha)s^{\alpha-1}},\qquad s>0,
\]
is not completely monotone on $(0,\infty)$.
\end{proposition}

\begin{proof}
Recall that if $g$ is completely monotone on $(0,\infty)$, then
$(-1)^n g^{(n)}(s)\ge 0$ for all $n\in\mathbb{N}_0$ and $s>0$.
In particular, taking $n=1$ yields $-g'(s)\ge 0$, hence every completely
monotone function is non-increasing on $(0,\infty)$; see
\cite[Theorem~1.4]{SchillingSongVondracek2012}.

A direct computation gives, for $s>0$,
\[
h_\alpha'(s)
=
\frac{(1-\alpha)^2, s^{\alpha-2}}
{\bigl(1+(1-\alpha)s^{\alpha-1}\bigr)^2}.
\]
Since $0<\alpha<1$, we have $(1-\alpha)^2>0$ and $s^{\alpha-2}>0$ for all
$s>0$, and the denominator is strictly positive. Therefore $h_\alpha'(s)>0$
for every $s>0$, i.e.\ $h_\alpha$ is strictly increasing on $(0,\infty)$.
This contradicts the necessary monotonicity property of completely monotone
functions. Hence $h_\alpha$ cannot be completely monotone on $(0,\infty)$.
\end{proof}

Figure~\ref{fig:halpha} illustrates that $h_\alpha$ is increasing for
representative values of $\alpha$, supporting Proposition~\ref{prop:h_not_cm}.

\begin{figure}[H]
\centering
\includegraphics[width=0.75\textwidth]{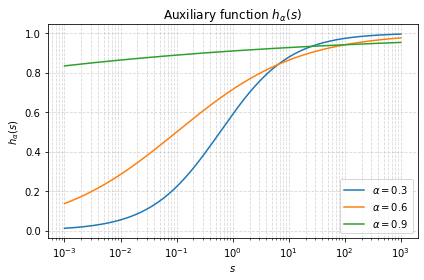}
\caption{Auxiliary function
$h_{\alpha}(s)=1/(1+(1-\alpha)s^{\alpha-1})$ for
$\alpha\in{0.3,0.6,0.9}$.
The monotone increase confirms Proposition~\ref{prop:h_not_cm}.}
\label{fig:halpha}
\end{figure}

Proposition~\ref{prop:h_not_cm} shows that the canonical product route based
on the factorization
\[
\Phi_{\alpha,\beta}(s)=s^\alpha h_\alpha(s)^\beta
\]
cannot be used to establish the Bernstein property of $\Phi_{\alpha,\beta}$
for $\beta>0$.

\begin{corollary}[Failure of the canonical Bernstein product mechanism]
\label{cor:Phi_not_Bernstein_mechanism}
For every $0<\alpha<1$ and $\beta>0$, the representation
\[
\Phi_{\alpha,\beta}(s)=s^\alpha h_\alpha(s)^\beta
\]
does not yield a Bernstein symbol through the standard product-type
construction based on a completely monotone auxiliary factor.
\end{corollary}

\begin{remark}
Corollary~\ref{cor:Phi_not_Bernstein_mechanism} should not be interpreted as a
non-Bernstein result for the symbol itself. It only shows that the most direct
factorization argument fails. The actual Bernstein property of
$\Phi_{\alpha,\beta}$ requires a separate analysis of the derivative
$\Phi'_{\alpha,\beta}$, which is carried out next.
\end{remark}

\begin{remark}[Origin of the complementary Bernstein range]
The obstruction to the Bernstein property for $\beta>1$ is proved in
Proposition~\ref{prop:Phi_not_BF_beta_gt_1}. The complementary positive range
$0\le\beta\le1$ was brought to the author's attention by the public staged
proof APP-0019 of the Pudim AI Project~\cite{PudimAI_APP0019_2026}. The
present manuscript includes a self-contained proof of this positive range in
Theorem~\ref{thm:Phi_BF_beta_le_1}. Together, these results give the sharp
Bernstein threshold
\[
\Phi_{\alpha,\beta}^{W}\in\mathcal{BF}
\quad\Longleftrightarrow\quad
0\le\beta\le1.
\]
This should not be confused with the complete Bernstein property. In
particular, complete Bernstein and subordination issues require additional
arguments and are not recovered from the Bernstein property alone.
\end{remark}

\subsection{Bernstein property in the natural range $0\le\beta\le1$}

\begin{theorem}[Bernstein property for $0\le\beta\le1$]
\label{thm:Phi_BF_beta_le_1}
Let $0<\alpha<1$ and $0\le\beta\le1$. Then
\[
\Phi_{\alpha,\beta}(s)
=
\frac{s^\alpha}{\bigl(1+(1-\alpha)s^{\alpha-1}\bigr)^\beta},
\qquad s>0,
\]
belongs to $\mathcal{BF}$.
\end{theorem}

\begin{proof}
Set
\[
a:=1-\alpha\in(0,1),
\qquad
y=s^a.
\]
Then
\[
\Phi_{\alpha,\beta}(s)
=s^{1-a}\bigl(1+a s^{-a}\bigr)^{-\beta}
=s^{1-a+a\beta}(s^a+a)^{-\beta}.
\]
Let
\[
p:=1-a+a\beta=\alpha+a\beta.
\]
A direct differentiation gives
\begin{align*}
\Phi'_{\alpha,\beta}(s)
&=p s^{p-1}(s^a+a)^{-\beta}
-\beta a s^{p+a-1}(s^a+a)^{-\beta-1} \\
&=s^{p-1}(s^a+a)^{-\beta-1}
\left[p(s^a+a)-\beta a s^a\right] \\
&=s^{p-1}(s^a+a)^{-\beta-1}
\left[\alpha s^a+\alpha a+a^2\beta\right] \\
&=s^{p-1}(s^a+a)^{-\beta}
\left(\alpha+\frac{a^2\beta}{s^a+a}\right).
\end{align*}
Since $p-1=a(\beta-1)$, we obtain
\[
\Phi'_{\alpha,\beta}(s)=H_{\alpha,\beta}(s^a),
\]
where
\[
H_{\alpha,\beta}(y)
=
y^{\beta-1}(y+a)^{-\beta}
\left(\alpha+\frac{a^2\beta}{y+a}\right),
\qquad y>0.
\]

We now prove that $H_{\alpha,\beta}$ is completely monotone for
$0\le\beta\le1$. If $0\le\beta<1$, then
$y^{\beta-1}=y^{-(1-\beta)}$ is completely monotone. Indeed, for
$r=1-\beta>0$,
\[
y^{-r}
=\frac{1}{\Gamma(r)}
\int_0^\infty e^{-yt}t^{r-1},dt.
\]
For $\beta=1$, this factor is simply equal to $1$.

For $\beta>0$, the shifted factor $(y+a)^{-\beta}$ is completely monotone,
since
\[
(y+a)^{-\beta}
=
\frac{1}{\Gamma(\beta)}
\int_0^\infty e^{-yt}e^{-at}t^{\beta-1},dt.
\]
For $\beta=0$, it is the constant $1$. Finally,
\[
\alpha+\frac{a^2\beta}{y+a}
\]
is the sum of a positive constant and a nonnegative multiple of the completely
monotone resolvent $(y+a)^{-1}$. By closure of completely monotone functions
under products and nonnegative sums, $H_{\alpha,\beta}$ is completely monotone.

The function $g(s)=s^a$ is a Bernstein function for $0<a<1$, since
\[
g'(s)=a s^{a-1}=a s^{-(1-a)}
\]
is completely monotone. The standard composition theorem for completely
monotone functions states that if $H$ is completely monotone and $g$ is a
Bernstein function with positive range, then $H\circ g$ is completely
monotone. Hence
\[
\Phi'_{\alpha,\beta}=H_{\alpha,\beta}\circ g
\]
is completely monotone on $(0,\infty)$.

\noindent Since $\Phi_{\alpha,\beta}$ is nonnegative and $C^\infty$ on $(0,\infty)$,
this proves that $\Phi_{\alpha,\beta}\in\mathcal{BF}$.
\end{proof}

\begin{remark}[Endpoint $\beta=1$]
For $\beta=1$, the above proof gives the particularly transparent formula
\[
H_{\alpha,1}(y)
=\frac{\alpha}{y+a}
+
\frac{a^2}{(y+a)^2},
\]
which is visibly completely monotone. The case $\beta=0$ reduces to the
classical Caputo symbol $\Phi_{\alpha,0}(s)=s^\alpha$.
\end{remark}

\subsection{A genuine obstruction for $\beta>1$}

Beyond the failure of the product mechanism, one can show that for
$\beta>1$ the symbol $\Phi_{\alpha,\beta}$ is not Bernstein at all.

\begin{proposition}[Non-Bernstein property for $\beta>1$]
\label{prop:Phi_not_BF_beta_gt_1}
Let $0<\alpha<1$ and $\beta>1$. Then
\[
\Phi_{\alpha,\beta}(s)
=\frac{s^\alpha}{\bigl(1+(1-\alpha)s^{\alpha-1}\bigr)^\beta},
\qquad s>0,
\]
does not belong to $\mathcal{BF}$.
\end{proposition}

\begin{proof}
Bernstein functions are concave on $(0,\infty)$. Writing
$a:=1-\alpha\in(0,1)$, we have
\[
\Phi_{\alpha,\beta}(s)
=
s^{\alpha+a\beta}(s^a+a)^{-\beta}.
\]
Letting $s\to0^+$ yields
\[
\Phi_{\alpha,\beta}(s)
\sim
a^{-\beta}s^{p},
\qquad
p=\alpha+(1-\alpha)\beta.
\]
If $\beta>1$, then $p>1$, so $\Phi_{\alpha,\beta}$ is locally convex near
$0$, which contradicts the concavity of Bernstein functions.
\end{proof}

\begin{corollary}[Sharp Bernstein threshold]
\label{cor:sharp_BF_threshold}
For every $0<\alpha<1$ and $\beta\ge0$,
\[
\Phi_{\alpha,\beta}\in\mathcal{BF}
\quad\Longleftrightarrow\quad
0\le\beta\le1.
\]
\end{corollary}

\begin{proof}
The implication
$0\le\beta\le1\Rightarrow\Phi_{\alpha,\beta}\in\mathcal{BF}$
follows from Theorem~\ref{thm:Phi_BF_beta_le_1}. The reverse implication
follows from Proposition~\ref{prop:Phi_not_BF_beta_gt_1}.
\end{proof}

\subsection{Positioning of the $W$-operator}

The preceding results show that the $W$-operator has a more subtle position
with respect to Bernstein theory than suggested by the canonical product
factorization alone. The factorization
\[
\Phi_{\alpha,\beta}(s)=s^\alpha h_\alpha(s)^\beta
\]
is not compatible with the standard product mechanism when $\beta>0$,
because $h_\alpha$ is not completely monotone. Nevertheless, the symbol itself
is Bernstein exactly in the natural range $0\le\beta\le1$.

This distinction is important. The $W$-operator is not obtained from the usual
complete-Bernstein-times-completely-monotone construction, but in the range
$0\le\beta\le1$ its symbol still belongs to the Bernstein class through a
direct derivative argument. For $\beta>1$, the operator leaves the Bernstein
class. At the Volterra level, however, the kernel representation and the
reciprocal-kernel construction remain meaningful for all $\beta\ge0$.

Despite the failure of the canonical product mechanism, the $W$-operator
retains a robust analytical structure: it admits a Volterra convolution kernel
expressed in terms of Prabhakar Mittag--Leffler functions, satisfies a
fundamental theorem of calculus, has the correct classical limit as
$\alpha\to1$, and generates well-posed abstract Cauchy problems via
resolvent-based Laplace inversion in the natural range $0\le\beta\le1$; see
Sections~\ref{sec:volterra-inverse} and~\ref{sec:W-evolution}.

Compared with the Atangana--Baleanu operator \cite{Luchko2019AB}, the
$W$-operator avoids high-frequency saturation and preserves a clear spectral
separation between fractional and first-order regimes. This controlled
frequency behaviour motivates its construction and places it in a distinct
Volterra--Bernstein configuration: Bernstein in the sharp range
$0\le\beta\le1$, but not through the standard product mechanism.

\section{The $W$--operator: Volterra representation and inverse operator}
\label{sec:volterra-inverse}

This section realizes the symbol $\Phi_{\alpha,\beta}$ (Section~\ref{sec:symbol}) as a concrete Volterra time--operator and constructs its natural left-inverse. The construction is entirely Laplace-based and therefore does not require the Bernstein classification proved in Section~\ref{sec:positioning}. In particular, we obtain (i) a Caputo-type normalization in the Laplace domain, (ii) a causal convolution (memory) representation, (iii) explicit Prabhakar kernels for both the derivative and its inverse, and (iv) a fractional fundamental theorem of calculus.

\subsection{Laplace-domain definition and Volterra kernel}

\begin{definition}[$W$--fractional derivative]\label{def:W-derivative}
Let $0<\alpha<1$ and $\beta\ge0$. For sufficiently regular
$u\colon[0,\infty)\to X$, define ${}^{W}D_t^{\alpha,\beta}u$ by
\begin{equation}\label{eq:W-def-Laplace}
\mathcal{L}\bigl[{}^{W}D_t^{\alpha,\beta}u\bigr](s)
=
\Phi_{\alpha,\beta}(s)\widehat{u}(s)
-
\frac{\Phi_{\alpha,\beta}(s)}{s}\,u(0),
\qquad \Re s>0,
\end{equation}
where $\Phi_{\alpha,\beta}$ is given by \eqref{eq:Phi}.
\end{definition}

\medskip
\noindent The subtraction in \eqref{eq:W-def-Laplace} is the usual Caputo-type normalization. In particular, when $\beta=0$ we recover the Caputo derivative of order $\alpha$.

\begin{proposition}[Volterra representation]\label{prop:Volterra-representation}
If $u\in C^1([0,T];X)$, then for $t\in(0,T]$,
\[
{}^{W}D_t^{\alpha,\beta}u(t)
=
\int_0^t u'(s)\,w_{\alpha,\beta}(t-s)\,ds,
\]
where $w_{\alpha,\beta}\in L^1_{\mathrm{loc}}(0,\infty)$ is uniquely determined
by
\begin{equation}\label{eq:w-Laplace}
\mathcal{L}[w_{\alpha,\beta}](s)=\frac{\Phi_{\alpha,\beta}(s)}{s},
\qquad \Re s>0.
\end{equation}
\end{proposition}

\begin{proof}
Use $\widehat u(s)=s^{-1}u(0)+s^{-1}\mathcal{L}[u'](s)$ in
\eqref{eq:W-def-Laplace} to obtain
\[
\mathcal{L}\bigl[{}^{W}D_t^{\alpha,\beta}u\bigr](s)
=
\frac{\Phi_{\alpha,\beta}(s)}{s}\,\mathcal{L}[u'](s).
\]
Define $w_{\alpha,\beta}$ via \eqref{eq:w-Laplace} and apply the convolution theorem.
\end{proof}

\begin{proposition}[Explicit memory kernel]\label{prop:kernel-explicit}
For $0<\alpha<1$ and $\beta\ge0$,
\[
w_{\alpha,\beta}(t)
=
t^{-\alpha}\,
E^{\beta}_{1-\alpha,\,1-\alpha}\bigl(-(1-\alpha)t^{1-\alpha}\bigr),
\qquad t>0.
\]
\end{proposition}

\begin{proof}
From \eqref{eq:w-Laplace} and \eqref{eq:Phi} we have
\[
\frac{\Phi_{\alpha,\beta}(s)}{s}
=
s^{\alpha-1}\bigl(1+(1-\alpha)s^{\alpha-1}\bigr)^{-\beta}.
\]
The inverse Laplace transform is given by the standard Prabhakar transform formula \cite{Prabhakar1971,GarraGorenfloPolito2014}.
\end{proof}

\begin{lemma}[Positivity of the $W$-memory kernel]\label{lem:positivity_w}
Let $0<\alpha<1$ and $0\le \beta \le 1$. Then
\[
w_{\alpha,\beta}(t)>0 \qquad\text{for all } t>0.
\]
\end{lemma}

\begin{proof}
Set $\rho:=1-\alpha\in(0,1)$ and $c:=1-\alpha>0$. Then
\[
w_{\alpha,\beta}(t)= t^{\rho-1} E^{\beta}_{\rho,\rho}\!\bigl(-c\,t^{\rho}\bigr).
\]
For $\beta=0$ we use $E^{0}_{\rho,\rho}(z)=1/\Gamma(\rho)$ to get
$w_{\alpha,0}(t)=t^{\rho-1}/\Gamma(\rho)>0$. For $0<\beta\le1$, a complete monotonicity criterion for Prabhakar kernels yields that $t\mapsto t^{\rho-1}E^{\beta}_{\rho,\rho}(-t^\rho)$ is completely
monotone (hence nonnegative) when $0<\rho\le1$ and $0<\rho\beta\le\rho$ \cite{GarraGorenfloPolito2014}. The scaling $t\mapsto c^{1/\rho}t$ preserves the sign, hence $w_{\alpha,\beta}(t)\ge0$. Since it is analytic and not identically zero, the inequality is strict for $t>0$.
\end{proof}

\medskip
\noindent This parameter restriction is consistent with the evolution theory in Section~\ref{sec:W-evolution}, where $0\le\beta\le1$ ensures the availability of robust kernel/resolvent estimates.

\begin{remark}[On the range $\beta>1$]\label{rem:beta_range_w}
Lemma~\ref{lem:positivity_w} covers the parameter range naturally compatible with known positivity/complete monotonicity criteria for Prabhakar-type kernels. For $\beta>1$, these arguments no longer control the sign, and the positivity of $w_{\alpha,\beta}$ would require a separate analysis.
\end{remark}

\subsection{Inverse Volterra kernel and the $W$--fractional integral}

\begin{lemma}[Inverse kernel: existence, formula, and local bounds]
\label{lem:inverse-kernel}
Let $0<\alpha<1$ and $\beta\ge0$. There exists $k_{\alpha,\beta}\in L^1_{\mathrm{loc}}(0,\infty)$ such that
\begin{equation}\label{eq:k-Laplace}
\mathcal{L}[k_{\alpha,\beta}](s)=\frac{1}{\Phi_{\alpha,\beta}(s)},
\qquad \Re s>0.
\end{equation}
Moreover,
\[
k_{\alpha,\beta}(t)
=
t^{\alpha-1}\,
E^{-\beta}_{1-\alpha,\,\alpha}\!\bigl(-(1-\alpha)t^{1-\alpha}\bigr),
\qquad t>0,
\]
and there exists $C>0$ such that
\begin{equation}\label{eq:k-local-bounds}
|k_{\alpha,\beta}(t)|\le C\,t^{\alpha-1}\qquad (0<t\le1),
\qquad
k_{\alpha,\beta}\in L^1_{\mathrm{loc}}(0,\infty).
\end{equation}
\end{lemma}

\begin{proof}
Starting from
\[
\frac{1}{\Phi_{\alpha,\beta}(s)}
=
s^{-\alpha}\bigl(1+(1-\alpha)s^{\alpha-1}\bigr)^{\beta},
\]
expand $(1+z)^\beta=\sum_{n\ge0}\binom{\beta}{n}z^n$ with
$z=(1-\alpha)s^{\alpha-1}$ and invert termwise using
$\mathcal{L}^{-1}[s^{-\rho}](t)=t^{\rho-1}/\Gamma(\rho)$ for $\rho>0$. This yields the stated Prabhakar representation and \eqref{eq:k-Laplace}. From the series definition of $E^{-\beta}_{1-\alpha,\alpha}$ one has $k_{\alpha,\beta}(t)\sim t^{\alpha-1}/\Gamma(\alpha)$ as $t\to0^+$, giving the bound \eqref{eq:k-local-bounds} and local integrability near $0$. Local integrability on $(1,T)$ follows from standard Prabhakar asymptotics; see, e.g., \cite{KilbasSrivastavaTrujillo2006,Mainardi2010}.
\end{proof}

\begin{remark}[On non-integer $\beta$]\label{rem:beta_noninteger}
For $\beta\notin\mathbb{N}$ the binomial series is understood as an analytic expansion for $|z|<1$; since both sides define analytic functions on $\mathbb{C}\setminus(-\infty,0]$, the identity extends to $\Re s>0$ by analytic continuation, which is sufficient for Laplace inversion.
\end{remark}

\begin{definition}[$W$--fractional integral]\label{def:W-integral}
Let $0<\alpha<1$ and $\beta\ge0$, and let $k_{\alpha,\beta}$ be given by Lemma~\ref{lem:inverse-kernel}. For $f\in L^1_{\mathrm{loc}}(0,\infty)$, define
\[
({}^{W}I_t^{\alpha,\beta}f)(t)
=
\int_0^t k_{\alpha,\beta}(t-s)\,f(s)\,ds,
\qquad t>0,
\]
whenever the convolution is finite.
\end{definition}

\begin{remark}[Classical limits]\label{rem:classical-limits}
If $\beta=0$, then $k_{\alpha,0}(t)=t^{\alpha-1}/\Gamma(\alpha)$ and
${}^{W}I^{\alpha,0}$ coincides with the Riemann--Liouville fractional integral. For $\beta>0$, the Prabhakar factor modifies memory while preserving the Volterra (causal) structure.
\end{remark}

\subsection{Fractional fundamental theorem of calculus}

\begin{theorem}[Fractional fundamental theorem of calculus]\label{thm:FTC}
Let $0<\alpha<1$ and $\beta\ge0$. Assume $f\in L^1_{\mathrm{loc}}(0,\infty)$ and
that $u:={}^{W}I_t^{\alpha,\beta}f$ is locally absolutely continuous with
$u(0)=0$. Then
\[
{}^{W}D_t^{\alpha,\beta}u(t)=f(t)
\quad\text{for almost every }t>0.
\]
\end{theorem}

\begin{proof}
From $u=k_{\alpha,\beta}*f$ and \eqref{eq:k-Laplace},
\[
\widehat{u}(s)=\frac{1}{\Phi_{\alpha,\beta}(s)}\,\widehat{f}(s),
\qquad \Re s>0.
\]
Applying \eqref{eq:W-def-Laplace} and using $u(0)=0$ gives
\[
\mathcal{L}\bigl[{}^{W}D_t^{\alpha,\beta}u\bigr](s)
=
\Phi_{\alpha,\beta}(s)\widehat{u}(s)
=
\widehat{f}(s).
\]
Injectivity of the Laplace transform yields ${}^{W}D_t^{\alpha,\beta}u=f$ a.e.
\end{proof}

\begin{remark}[Independence of Bernstein theory]\label{rem:no_Bernstein_needed}
The Volterra representation, the inverse operator, and
Theorem~\ref{thm:FTC} follow from reciprocal Laplace symbols and convolution alone; no complete monotonicity or subordination hypothesis is required. This remains true independently of the sharp Bernstein threshold established in Section~\ref{sec:positioning}.
\end{remark}

\section{Abstract $W$-fractional evolution problems and well-posedness}
\label{sec:W-evolution}

The Volterra representation and Laplace-symbol analysis developed in the previous sections naturally lead to the study of abstract evolution equations driven by the $W$-operator. In this section, we establish a well-posedness theory for the $W$-fractional Cauchy problem in Banach spaces, based on Laplace-domain resolvent representations and sectorial functional calculus.

\medskip
\noindent We consider the abstract problem
\begin{equation}\label{eq:W-ACP}
\begin{cases}
{}^{W}D_{t}^{\alpha,\beta} u(t) + A u(t) = f(t), & t>0,\\[0.3em]
u(0)=u_0\in X,
\end{cases}
\end{equation}
where $0<\alpha<1$, $0\le\beta\le1$, $X$ is a complex Banach space, and $A$ is a (possibly unbounded) linear operator on $X$.

\subsection{Laplace-domain formulation and sectorial framework}

\noindent Throughout this section we assume that $A$ is sectorial of angle $\varphi<\pi/2$, that is,
\[
  \sigma(A)\subset\overline{\Sigma_\varphi},
  \qquad
  \|(zI-A)^{-1}\|\le \frac{C_\psi}{|z|}
  \quad\text{for all }z\notin \Sigma_\psi,
\]
for every $\psi\in(\varphi,\pi)$. Equivalently, $-A$ generates a bounded analytic $C_0$-semigroup on $X$.

\noindent Taking Laplace transforms in \eqref{eq:W-ACP} yields the algebraic identity
\begin{equation}\label{eq:uhat-W}
  \widehat{u}(s)
  = (\Phi_{\alpha,\beta}(s)I+A)^{-1}
    \Bigl(
      \widehat{f}(s)
      + \Phi_{\alpha,\beta}(s)s^{-1}u_0
    \Bigr),
  \qquad \Re s>0.
\end{equation}
This representation motivates the definition of the associated resolvent family via inverse Laplace transform.

\subsection{Well-posedness and $W$-resolvent family}

\begin{theorem}[Well-posedness of the $W$-fractional evolution problem] \label{thm:W-wellposed}
Let $0<\alpha<1$, $0\le\beta\le1$, and let $A$ be sectorial of angle
$<\pi/2$ on $X$. Fix $\theta\in(\varphi,\pi)$ and let $\Gamma_\theta$ denote a standard contour in $\mathbb{C}\setminus\overline{\Sigma_\theta}$.

\begin{enumerate}[label=\textnormal{(\roman*)}]
\item \textbf{Resolvent estimate.}
There exists $C_\theta>0$ such that
\begin{equation}\label{eq:resolvent-bound-Phi}
\|(\Phi_{\alpha,\beta}(s)I + A)^{-1}\|
\le \frac{C_\theta}{|\Phi_{\alpha,\beta}(s)|},
\qquad s\in\Gamma_\theta .
\end{equation}
In particular, using Proposition~\ref{prop:Phi_2regime},
\[
\|(\Phi_{\alpha,\beta}(s)I + A)^{-1}\|
\lesssim
\begin{cases}
|s|^{-\alpha}, & |s|\ge1,\\[0.2em]
|s|^{-(\alpha+\beta(1-\alpha))}, & 0<|s|\le1,
\end{cases}
\qquad s\in\Gamma_\theta .
\]

\item \textbf{Definition of the $W$-resolvent family.}
For $t>0$, define
\begin{equation}\label{eq:W-family}
W_{\alpha,\beta}(t)
:= \frac{1}{2\pi i}
   \int_{\Gamma_\theta}
   e^{st}(\Phi_{\alpha,\beta}(s)I + A)^{-1}\,ds .
\end{equation}
Then $W_{\alpha,\beta}(t)\in\mathcal{L}(X)$ and the integral converges absolutely.

\item \textbf{Mild solution.}
For every $u_0\in X$ and $f\in L^1_{\rm loc}([0,T];X)$, the function
\begin{equation}\label{eq:mild-solution}
u(t)= W_{\alpha,\beta}(t)u_0
      +\int_0^t W_{\alpha,\beta}(t-\tau)\,f(\tau)\,d\tau
\end{equation}
belongs to $C((0,T];X)$, satisfies $u(t)\to u_0$ as $t\downarrow0$, and solves \eqref{eq:W-ACP} in the sense of Laplace transforms.

\item \textbf{Uniqueness.}
The mild solution is unique in $C((0,T];X)\cap L^1_{\rm loc}((0,T];X)$.
\end{enumerate}
\end{theorem}

\begin{proof}
We work throughout under the standing assumption $0\le\beta\le1$, which is the natural parameter range for the evolution theory developed in this section.

\textit{(i)} Since $A$ is sectorial of angle $<\pi/2$, we may choose
$\theta\in(\varphi,\pi)$ so that the sectorial resolvent estimate holds outside $\Sigma_\theta$. For $s\in\Gamma_\theta$, the complex number $\lambda=\Phi_{\alpha,\beta}(s)$ stays in a region avoiding $-\sigma(A)$ (by the sectorial geometry and the fact that $\Gamma_\theta$ avoids the negative real axis). Hence
\[
\|(\Phi_{\alpha,\beta}(s)I + A)^{-1}\|
= \|(\lambda I + A)^{-1}\|
\le \frac{C_\theta}{|\lambda|}
= \frac{C_\theta}{|\Phi_{\alpha,\beta}(s)|}.
\]
The two-regime bounds then follow directly from Proposition~\ref{prop:Phi_2regime}.

\textit{(ii)} Absolute convergence of \eqref{eq:W-family} follows from the decay of $e^{st}$ along $\Gamma_\theta$ together with \eqref{eq:resolvent-bound-Phi} and the two-regime estimate.

\textit{(iii)} Taking Laplace transforms in \eqref{eq:mild-solution} and using Tonelliâ€™s theorem yields \eqref{eq:uhat-W}. In particular, the contour-defined family \eqref{eq:W-family} is well defined and its Laplace transform coincides with the resolvent $(\Phi_{\alpha,\beta}(s)I+A)^{-1}$ for $\Re s>0$, so that
\eqref{eq:mild-solution} is equivalent to \eqref{eq:W-ACP} in the sense of Laplace transforms. Continuity on $(0,T]$ follows from dominated convergence applied to the contour integral defining $W_{\alpha,\beta}(t)$, while the limit $u(t)\to u_0$ as
$t\downarrow0$ is obtained from the standard Laplace inversion/resolvent-family principle for sectorial generators; see, e.g., \cite[Ch.~3]{Pruss2016} or \cite[Ch.~2]{Haase2006}.

\textit{(iv)} Uniqueness follows by subtraction of two mild solutions and Laplace transform injectivity.
\end{proof}

\begin{definition}[$W$-resolvent family]
The family $\{W_{\alpha,\beta}(t)\}_{t>0}$ defined by \eqref{eq:W-family} is called the \emph{$W$-resolvent family} associated with the operator ${}^{W}D_{t}^{\alpha,\beta}$ and the generator $A$.
\end{definition}

\begin{proposition}[Temporal continuity and local H\"older regularity]
\label{prop:temporal-regularity}
Let $A$ be a sectorial operator of angle $<\pi/2$ on a complex Banach space $X$. Let $0<\alpha<1$ and $0\le\beta\le1$, and let $u_0\in X$ and $f\in L^1_{\mathrm{loc}}([0,T];X)$. Let $u$ be the mild solution of \eqref{eq:W-ACP} given by
\[
u(t)=W_{\alpha,\beta}(t)u_0
+\int_0^t W_{\alpha,\beta}(t-s)\,f(s)\,ds,
\qquad t\in(0,T].
\]
Then $u\in C((0,T];X)$ and
\[
\lim_{t\downarrow0} u(t)=u_0 \quad \text{in } X.
\]

Moreover, if $f\in C([0,T];X)$, then for every $0<t_0<T$ the map
$t\mapsto u(t)$ is H\"older continuous on $[t_0,T]$: for any $\eta\in(0,1)$ there exists $C_{\eta,t_0,T}>0$ such that
\[
\|u(t+h)-u(t)\|
\le C_{\eta,t_0,T}\,|h|^\eta,
\qquad t,t+h\in[t_0,T].
\]
\end{proposition}
\begin{proof}
\textbf{Part 1: Strong continuity of the $W$-resolvent family.}
Let $x\in X$ and $t>0$. Using the contour representation
\[
W_{\alpha,\beta}(t)x
=\frac{1}{2\pi i}\int_{\Gamma_\theta}
e^{st}(\Phi_{\alpha,\beta}(s)I+A)^{-1}x\,ds,
\]
let $t_n\to t$ with $t_n\ge t/2$ for $n$ large. For each fixed $s\in\Gamma_\theta$, we have $e^{st_n}\to e^{st}$. Moreover, since $\Re(s)\le -c|s|$ on $\Gamma_\theta$ and
\[
\|(\Phi_{\alpha,\beta}(s)I+A)^{-1}\|
\lesssim \frac{1}{|\Phi_{\alpha,\beta}(s)|},
\]
the two-regime estimates of Proposition~\ref{prop:Phi_2regime} imply
\[
\|e^{st_n}(\Phi_{\alpha,\beta}(s)I+A)^{-1}x\|
\lesssim e^{-c|s|t/2}\,|s|^{-\alpha}\,\|x\|,
\]
which is integrable along $\Gamma_\theta$. By dominated convergence,
$W_{\alpha,\beta}(t_n)x\to W_{\alpha,\beta}(t)x$. Thus $t\mapsto W_{\alpha,\beta}(t)x$ is continuous on $(0,\infty)$.

\medskip
\textbf{Part 2: Continuity of $u$ on $(0,T]$.}
Let $t_n\to t\in(0,T]$. The convergence $W_{\alpha,\beta}(t_n)u_0\to W_{\alpha,\beta}(t)u_0$
follows from part~1.
For the convolution term, write
\[
\int_0^{t_n}W_{\alpha,\beta}(t_n-s)f(s)\,ds
-\int_0^t W_{\alpha,\beta}(t-s)f(s)\,ds
= I_1(n)+I_2(n),
\]
where
\[
I_1(n)=\int_0^t\bigl(W_{\alpha,\beta}(t_n-s)-W_{\alpha,\beta}(t-s)\bigr)f(s)\,ds,
\quad
I_2(n)=\int_t^{t_n}W_{\alpha,\beta}(t_n-s)f(s)\,ds.
\]
By part~1 and the bound $\|W_{\alpha,\beta}(\tau)\|\lesssim\tau^{\alpha-1}$, $I_1(n)\to0$ by dominated convergence. Moreover,
\[
\|I_2(n)\|
\lesssim \int_0^{|t_n-t|}\tau^{\alpha-1}\,d\tau
\longrightarrow0,
\]
hence $u\in C((0,T];X)$.

\medskip
\textbf{Part 3: Limit as $t\downarrow0$.}
By Theorem~\ref{thm:W-wellposed}, $W_{\alpha,\beta}(t)u_0\to u_0$ strongly as $t\downarrow0$. Since $W_{\alpha,\beta}\in L^1_{\mathrm{loc}}(0,\infty)$,
\[
\left\|\int_0^t W_{\alpha,\beta}(t-s)f(s)\,ds\right\|
\lesssim \int_0^t (t-s)^{\alpha-1}\|f(s)\|\,ds
\longrightarrow0,
\]
which proves $\lim_{t\downarrow0}u(t)=u_0$.

\medskip
\textbf{Part 4: Local H\"older regularity when $f\in C([0,T];X)$.}
Fix $0<t_0<T$ and $t,t+h\in[t_0,T]$. Write
\[
u(t+h)-u(t)
=\bigl(W_{\alpha,\beta}(t+h)-W_{\alpha,\beta}(t)\bigr)u_0
+J_1+J_2,
\]
with
\[
J_1=\int_0^t\bigl(W_{\alpha,\beta}(t+h-s)-W_{\alpha,\beta}(t-s)\bigr)f(s)\,ds,
\quad
J_2=\int_t^{t+h}W_{\alpha,\beta}(t+h-s)f(s)\,ds.
\]
Using the contour representation and the inequality
$|e^{z}-1|\le C_\eta |z|^\eta$ for $\eta\in(0,1)$, we obtain
\[
\|W_{\alpha,\beta}(t+h)-W_{\alpha,\beta}(t)\|
\lesssim |h|^\eta
\int_{\Gamma_\theta} e^{-c|s|t_0}|s|^{\eta-\alpha}\,|ds|
\le C|h|^\eta.
\]
The same estimate applies to $J_1$ by boundedness of $f$. For $J_2$ we use $\|W_{\alpha,\beta}(\tau)\|\lesssim\tau^{\alpha-1}$ to get
\[
\|J_2\|\lesssim \int_0^{|h|}\tau^{\alpha-1}\,d\tau
\lesssim |h|^\alpha
\le |h|^\eta,
\]
since $\eta<1$. Combining the estimates yields the H\"older bound.
\end{proof}

\subsection{Regularity, smoothing, and parameter stability}

Assume $0\le\beta\le1$ and let $A$ be sectorial (analytic) as above.
If $A$ is analytic, the $W$-resolvent family exhibits fractional smoothing of the standard form: for $\gamma\in[0,1]$ one expects bounds
\[
\|A^\gamma W_{\alpha,\beta}(t)\|
\lesssim t^{-\alpha\gamma},
\qquad t>0,
\]
obtained by combining \eqref{eq:resolvent-bound-Phi} with functional calculus estimates for sectorial operators and classical contour arguments.

\noindent Moreover, for every $x\in X$ and $0<t_0<T<\infty$, the map
$(\alpha,\beta,t)\mapsto W_{\alpha,\beta}(t)x$ is continuous on compact subsets of $\{0<\alpha<1,\ 0\le\beta\le1\}\times[t_0,T]$ by dominated convergence in the contour representation \eqref{eq:W-family}.

\noindent Finally, the $W$-evolution problem is stable under classical limits within this parameter regime:
\begin{itemize}
\item as $\alpha\to1$, one recovers the classical analytic semigroup solution of $u'(t)+Au(t)=f(t)$;
\item as $\beta\to0$, one recovers the Caputo fractional solution of order $\alpha$.
\end{itemize}
The convergence holds pointwise in time and uniformly on every interval $[t_0,T]$ with $t_0>0$.

\section{Memory kernels: qualitative comparison}
\label{sec:kernels-comparison}

This section provides a qualitative and illustrative comparison of memory kernels, aimed at highlighting structural differences rather than delivering a full numerical analysis.

\noindent A key motivation for introducing ${}^{W}D_{t}^{\alpha,\beta}$ is to retain an explicit Volterra structure while allowing a controlled modification of the memory profile through the parameter $\beta$.

\subsection{Caputo kernel}

The Caputo derivative of order $\alpha\in(0,1)$ can be written as
\[
    {}^{\mathrm{C}}D_t^\alpha u(t)
    = \int_0^t u'(s)\,k_{\mathrm{C},\alpha}(t-s)\,ds,
\]
with kernel
\[
    k_{\mathrm{C},\alpha}(t)
    = \frac{1}{\Gamma(1-\alpha)}\,t^{-\alpha},
    \qquad t>0.
\]
This kernel is strongly singular at $t=0$, reflecting a pronounced short-time memory effect; see \cite{Podlubny1999,Mainardi2010}.

\subsection{AB kernel (non-singular but structurally non-Volterra)}

In the AB case, the kernel is of Mittag--Leffler type and behaves like a bounded function near $t=0$, which removes the singularity but breaks the classical Volterra--Bernstein structure: the symbol is not Bernstein and the kernel is not completely monotone \cite{Luchko2019AB,Giusti2018AB}. Graphically, one observes a
smooth onset of memory with exponential-type decay and saturation effects at high frequencies.

\subsection{$W$-kernel}

For the $W$-operator, the Volterra kernel is given explicitly by
Proposition~\ref{prop:kernel-explicit}:
\[
    w_{\alpha,\beta}(t)
    =
    t^{-\alpha}\,
    E^{\beta}_{1-\alpha,\,1-\alpha}\!\bigl(-(1-\alpha)\,t^{1-\alpha}\bigr),
    \qquad t>0.
\]
For $\beta=0$, this reduces to the Caputo kernel (up to normalization). For $\beta>0$, the Prabhakar factor modifies the memory profile while preserving a fully explicit Laplace/Volterra structure compatible with sectorial operator theory. In particular, the two-regime behaviour of $\Phi_{\alpha,\beta}$ from Proposition~\ref{prop:Phi_2regime} translates into a tunable long-time regime without imposing high frequency saturation.

\medskip
\noindent Figure~\ref{fig:kernels-comparison} provides a qualitative comparison of the Caputo, AB, and $W$-memory kernels for a fixed $\alpha$ and representative values of $\beta$.

\begin{figure}[H]
    \centering
    \includegraphics[width=1.0\textwidth]{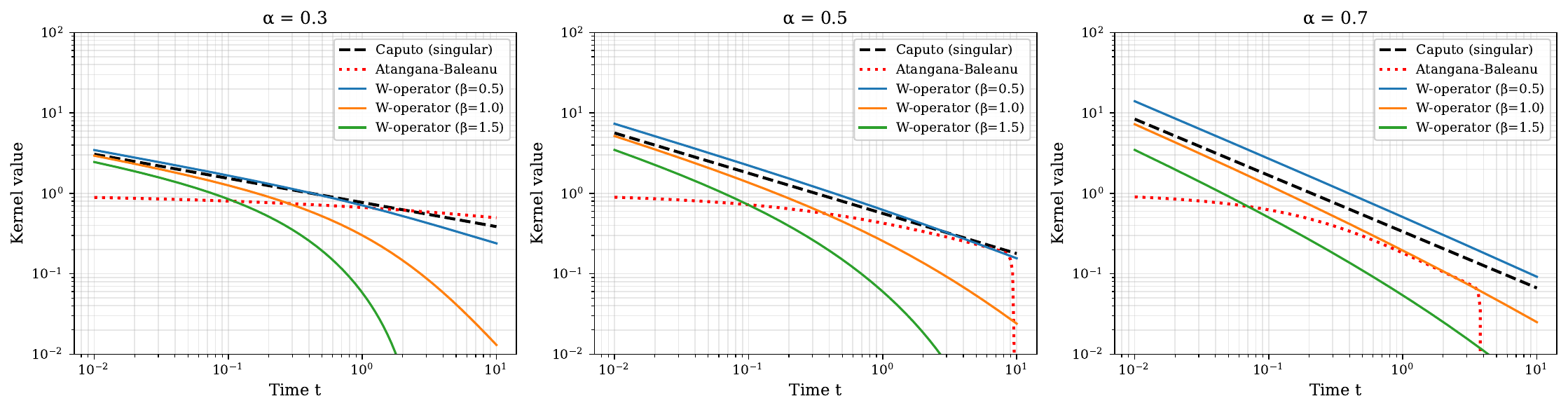}
    \caption{Qualitative comparison of the kernels
    $k_{\mathrm{C},\alpha}(t)$ (Caputo), $k_{\mathrm{AB},\alpha}(t)$
    (Atangana--Baleanu), and $w_{\alpha,\beta}(t)$ ($W$-operator) for a fixed
    $\alpha\in(0,1)$ and representative values of $\beta$.
    The Caputo kernel is singular at $t=0$, while AB kernels are bounded at the
    origin but typically exhibit saturation effects. The $W$-kernel preserves a
    transparent Volterra/Laplace structure and allows controlled modulation of
    the memory profile through $\beta$.}
    \label{fig:kernels-comparison}
\end{figure}

\noindent A more detailed numerical study of $w_{\alpha,\beta}(t)$, including its monotonicity and convexity properties and its dependence on $(\alpha,\beta)$, would provide additional insight into the memory profile encoded by ${}^{W}D_{t}^{\alpha,\beta}$ and help position it with respect to Caputo and AB models.
\section{A model application: $W$-fractional diffusion}
\label{sec:model-application}

To illustrate the scope and applicability of the $W$-operator
${}^{W}D_{t}^{\alpha,\beta}$, we briefly consider a model diffusion
problem driven by this fractional time-derivative.
Let $\Omega\subset\mathbb{R}^d$ be a bounded domain with smooth boundary, and consider
\begin{equation}\label{eq:W-diffusion}
    {}^{W}D_{t}^{\alpha,\beta} u(x,t)
    = \Delta u(x,t) + f(x,t),
    \qquad x\in\Omega,\ t>0,
\end{equation}
subject to homogeneous Dirichlet boundary conditions
\[
    u(x,t) = 0, \qquad x\in\partial\Omega,\ t>0,
\]
and initial condition
\[
    u(x,0) = u_0(x), \qquad x\in\Omega.
\]

\noindent This problem serves as a canonical example of a linear parabolic equation with memory, allowing a direct application of the abstract theory developed in Section~\ref{sec:W-evolution}.

\subsection{Relation with classical fractional diffusion}

For $\beta=0$, the symbol reduces to $\Phi_{\alpha,0}(s)=s^\alpha$ and \eqref{eq:W-diffusion} coincides with the classical time-fractional diffusion equation in the Caputo sense,
\[
    {}^{\mathrm{C}}D_t^\alpha u(x,t)
    = \Delta u(x,t) + f(x,t),
\]
for which a well-established theory is available, including
well-posedness, subordination formulas, smoothing estimates, and
numerical schemes based on convolution quadrature; see, for instance,
\cite{Pruss2016,SakamotoYamamoto2011,LiZhang2019}.

For $\beta>0$, the symbol $\Phi_{\alpha,\beta}$ preserves the Caputo-type high-frequency behavior $\Phi_{\alpha,\beta}(s)\sim s^\alpha$ while modifying the low-frequency regime. In the natural range $0\le\beta\le1$, Theorem~\ref{thm:Phi_BF_beta_le_1} shows that this symbol is Bernstein, although not by the canonical product mechanism. Mild solutions of \eqref{eq:W-diffusion} are constructed directly through Laplace inversion and resolvent estimates, working with the transformed resolvent
\[
    (\Phi_{\alpha,\beta}(s)I-\Delta)^{-1}
\]
along suitable sectorial contours; see \cite{Haase2006,Pruss2016}. This direct Volterra-resolvent construction is independent of any stronger complete-Bernstein or subordination interpretation.

\noindent As a direct consequence of the abstract well-posedness result Theorem~\ref{thm:W-wellposed} in Section~\ref{sec:W-evolution}, we obtain the following result.

\begin{corollary}[Well-posedness of $W$-fractional diffusion]
\label{cor:W-diffusion}
Let $0<\alpha<1$ and $0\le\beta\le1$. Let $X=L^2(\Omega)$ and let
$A=-\Delta$ with homogeneous Dirichlet boundary conditions on $\Omega$. Then $-A$ generates a bounded analytic $C_0$-semigroup on $L^2(\Omega)$, and the problem \eqref{eq:W-diffusion} admits a unique mild solution $u\colon(0,\infty)\to L^2(\Omega)$ for each $u_0\in L^2(\Omega)$ and $f\in L^1_{\mathrm{loc}}(0,\infty;L^2(\Omega))$.
This solution is given by the Laplace-resolvent representation associated with the $W$-resolvent family and depends continuously on the data $(u_0,f)$.
\end{corollary}

\begin{proof}
It is classical that $-A=-\Delta$ with homogeneous Dirichlet boundary
conditions generates a bounded analytic $C_0$-semigroup on $L^2(\Omega)$ and that $A$ is sectorial of angle $<\pi/2$; see, e.g., \cite{Pruss2016}. The conclusion then follows directly from Theorem~\ref{thm:W-wellposed}.
\end{proof}

\subsection{Heuristic comparison of the regimes ${\beta=0}$ and ${\beta=1}$}

The two distinguished cases
\[
    \beta=0
    \quad\text{and}\quad
    \beta=1
\]
illustrate the qualitative role of the parameter $\beta$.
For $\beta=0$, the dynamics corresponds to classical fractional
subdiffusion of order $\alpha$. For $\beta=1$, the symbol
\[
    \Phi_{\alpha,1}(s)
    = \frac{s}{s^{1-\alpha}+(1-\alpha)}
\]
suggests a competition between fractional and first-order behavior, with a crossover between low- and high-frequency regimes.

\noindent In this case, $\beta=1$ lies at the endpoint of the Bernstein range (Theorem~\ref{thm:Phi_BF_beta_le_1}). The standard Bernstein product mechanism still fails (Proposition~\ref{prop:h_not_cm} and Corollary~\ref{cor:Phi_not_Bernstein_mechanism}), yet the evolution problem remains well posed. From a qualitative viewpoint, one expects:
\begin{itemize}
    \item diffusion slower than the classical heat equation due to the presence of memory, but potentially faster than in the pure Caputo case ($\beta=0$);
    \item a regularized memory effect, without artificial saturation,
    while preserving a Volterra convolution structure;
    \item intermediate decay rates for the energy $\|u(\cdot,t)\|_{L^2(\Omega)}^2$ between the Caputo subdiffusive regime and the standard parabolic case.
\end{itemize}
The $W$-fractional diffusion model therefore provides a flexible framework for exploring memory effects that interpolate between fractional and classical diffusion while remaining compatible with resolvent-based operator theory.

\begin{remark}[Spectral interpretation and modal behavior]
\label{rem:spectral-behavior}
Let $(\lambda_k,\varphi_k)_{k\ge1}$ denote the eigenpairs of the Dirichlet Laplacian $-\Delta$ on $\Omega$, with $0<\lambda_1\le\lambda_2\le\cdots$ and $\{\varphi_k\}$ forming an orthonormal basis of $L^2(\Omega)$. Formally expanding the solution of \eqref{eq:W-diffusion} as
\[
u(x,t)=\sum_{k\ge1} u_k(t)\,\varphi_k(x),
\]
each modal coefficient $u_k$ satisfies the scalar evolution equation
\[
{}^{W}D_t^{\alpha,\beta} u_k(t) + \lambda_k u_k(t) = f_k(t),
\qquad
u_k(0)=\langle u_0,\varphi_k\rangle.
\]

In the Laplace domain this yields
\[
\widehat{u_k}(s)
=
\frac{\Phi_{\alpha,\beta}(s)}{s\bigl(\Phi_{\alpha,\beta}(s)+\lambda_k\bigr)}
\,u_k(0)
+
\frac{1}{\Phi_{\alpha,\beta}(s)+\lambda_k}\,\widehat{f_k}(s),
\]
which makes explicit the role of the symbol $\Phi_{\alpha,\beta}$ as a frequency-dependent damping mechanism.

\noindent For large eigenvalues $\lambda_k$ (high spatial frequencies), the
dominant balance is governed by the high-frequency behavior
$\Phi_{\alpha,\beta}(s)\sim s^\alpha$, and the decay of each mode is
therefore comparable to that of the classical Caputo fractional diffusion. By contrast, for low eigenvalues $\lambda_k$ (large-scale modes), the modified low-frequency regime of $\Phi_{\alpha,\beta}$ becomes relevant. In particular, for $\beta=1$ the linearization of $\Phi_{\alpha,1}$ near $s=0$ induces a transition toward a first-order temporal behavior for the lowest modes.

\noindent As a consequence, the $W$-fractional diffusion model exhibits a scale-dependent relaxation mechanism: high-frequency modes retain fractional subdiffusive decay, while low-frequency modes may experience a faster, more classical-type relaxation. This spectral perspective provides an intuitive explanation for the intermediate energy decay rates discussed above and motivates the use of the parameter $\beta$ as a tunable interpolation between purely fractional and classical diffusion regimes.
\end{remark}

\medskip

\noindent This modal interpretation suggests that the parameter $\beta$ may serve as a physically meaningful tuning parameter for balancing long-range memory and classical relaxation, a feature that deserves further quantitative study.

\subsection{Perspective for numerical experiments}

A concrete numerical investigation of \eqref{eq:W-diffusion} could be
conducted along the following lines:
\begin{itemize}
    \item discretize space using a finite element or finite difference method, leading to a discrete Laplacian matrix $A_h$;
    \item approximate the operator ${}^{W}D_{t}^{\alpha,\beta}$ by a
    convolution quadrature scheme based on a rational approximation of the symbol $\Phi_{\alpha,\beta}$, in the spirit of
    \cite{Lubich1986,LiZhang2019};
    \item compare, for the same spatial mesh and time grid, the solutions corresponding to $\beta=0$ (Caputo), $\beta=1$ ($W$-operator), and the classical heat equation ($\alpha=1$);
    \item analyze the decay of the discrete energy
    $\|u_h(t)\|_{L^2(\Omega)}^2$ and the short-time behavior of the
    numerical solution.
\end{itemize}

\noindent Such experiments would provide a first illustration of the influence of the parameter $\beta$ on diffusion and memory effects, and would also serve as a practical validation of the numerical admissibility of the $W$-operator in a representative parabolic PDE setting.

\subsection{A numerical validation in 1D}

To validate the theoretical findings and illustrate the dissipative nature of the $W$-operator, we solve the diffusion equation~\eqref{eq:W-diffusion} on the domain $\Omega = (0,1)$ with homogeneous Dirichlet boundary conditions. The initial condition is set to the first eigenmode $u_0(x) = \sin(\pi x)$.

\noindent The numerical solution is computed using a spectral decomposition combined with a high-precision inverse Laplace transform method (fixed Talbot contour with $N=24$ quadrature points) to ensure stability and avoid numerical artifacts associated with singular kernels.

\medskip
\noindent Figure~\ref{fig:sensitivity} presents the time evolution of the normalized energy $E(t)/E(0)$ for three fractional orders $\alpha \in \{0.3, 0.5, 0.9\}$ and varying modulation parameters $\beta \in \{0.0, 0.3, 0.5, 0.7, 1.0\}$. The quantitative decay properties are summarized in Table~\ref{tab:sensitivity}.

\begin{figure}[H]
    \centering
    \includegraphics[width=\textwidth]{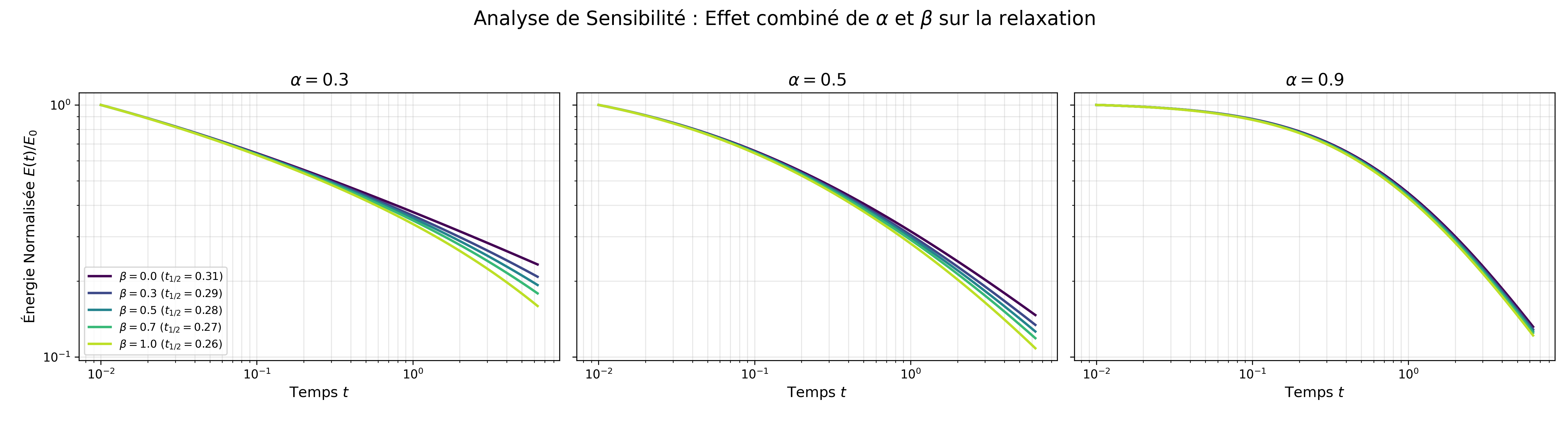}
    \caption{Sensitivity analysis: Combined effect of $\alpha$ and $\beta$ on the energy relaxation.}
    \label{fig:sensitivity}
\end{figure}

\noindent The results confirm that the energy decays monotonically for all parameter combinations, validating the well-posedness of the model. We observe two key trends:
\begin{enumerate}
    \item \textbf{Acceleration effect:} For a fixed $\alpha$, increasing $\beta$ systematically accelerates the decay (steeper slope and shorter half-life $t_{1/2}$).
    \item \textbf{Alpha-dependence:} The influence of $\beta$ is most pronounced at lower fractional orders (e.g., $\alpha=0.3$), where it significantly alters the decay slope. As $\alpha \to 1$, the diffusion process is dominated by the near-integer derivative, and $\beta$ acts primarily as a fine-tuning parameter for the transient regime.
\end{enumerate}

\begin{table}[H]
\centering
\begin{tabular}{|c|c|c|c|}
\hline
$\alpha$ & $\beta$ & Decay Slope & Half-life ($t_{1/2}$) \\
\hline
\hline
0.3 & 0.0 & -0.257 & 0.308 \\
0.3 & 0.3 & -0.292 & 0.292 \\
0.3 & 0.5 & -0.317 & 0.282 \\
0.3 & 0.7 & -0.341 & 0.273 \\
0.3 & 1.0 & -0.378 & 0.261 \\
\hline
0.5 & 0.0 & -0.403 & 0.265 \\
0.5 & 0.3 & -0.430 & 0.254 \\
0.5 & 0.5 & -0.448 & 0.247 \\
0.5 & 0.7 & -0.466 & 0.240 \\
0.5 & 1.0 & -0.494 & 0.231 \\
\hline
0.9 & 0.0 & -0.608 & 0.793 \\
0.9 & 0.3 & -0.614 & 0.775 \\
0.9 & 0.5 & -0.618 & 0.763 \\
0.9 & 0.7 & -0.622 & 0.751 \\
0.9 & 1.0 & -0.628 & 0.733 \\
\hline
\end{tabular}
\caption{Sensitivity analysis of decay properties for varying $\alpha$ and $\beta$.}
\label{tab:sensitivity}
\end{table}

\section{Conclusion}

We have introduced and analyzed a new two-parameter fractional time operator ${}^{W}D_{t}^{\alpha,\beta}$ with Volterra structure, designed to preserve the Caputo-type short-time behavior while allowing a controlled modification of the long-time memory through the parameter $\beta$. The operator is defined via an explicit Laplace symbol and admits a concrete realization as a causal convolution with Prabhakar-type kernels.

\noindent A central contribution of this work is the clarification of the structural position of the $W$-operator with respect to the classical Bernstein/subordination framework. We have shown that the natural factorization of the symbol fails to generate a Bernstein structure through the canonical product mechanism for $\beta>0$. However, this failure of the mechanism does not imply failure of the Bernstein property itself. By analyzing $\Phi'_{\alpha,\beta}$ directly, we proved the sharp threshold
\[
\Phi_{\alpha,\beta}\in\mathcal{BF}
\quad\Longleftrightarrow\quad
0\le\beta\le1.
\]
For $\beta>1$, the symbol is not Bernstein, as follows from the low-frequency asymptotic convexity obstruction. Thus, the natural evolution range $0\le\beta\le1$ is simultaneously the Bernstein range, the positive-kernel range, and the robust well-posedness range considered in this paper.

\noindent Independently of the Bernstein classification, the $W$-operator retains a robust analytical structure: it admits an explicit inverse Volterra operator, satisfies a fractional fundamental theorem of calculus, and fits naturally into a resolvent-based approach to abstract evolution equations. At the purely symbolic/Volterra level (Laplace reciprocity and convolution), the kernel representations and the left-inverse construction remain meaningful for all $\beta\ge0$.

\noindent We established well-posedness of abstract $W$-fractional Cauchy problems with sectorial generators in the range $0\le\beta\le1$, based on Laplace-domain resolvent estimates and contour inversion, leading to the construction of a $W$-resolvent family and a mild solution formula.

\noindent In addition, we proved temporal continuity on $(0,T]$ and local H\"older regularity away from $t=0$ under mild assumptions on the forcing term, and we discussed the expected fractional smoothing estimates for analytic generators. Finally, we recovered the classical limits $\alpha\to1$ and $\beta\to0$ on every interval $[t_0,T]$ with $t_0>0$.

\noindent The application to a $W$-fractional diffusion equation illustrates how the parameter $\beta$ induces a scale-dependent relaxation of spectral modes, interpolating between purely fractional and more classical diffusion regimes.

\noindent Several directions for future research naturally emerge from this work. These include a deeper spectral analysis of the $W$-resolvent family, rigorous parameter-continuity results in $(\alpha,\beta)$, possible complete-Bernstein refinements, extensions to nonlinear problems driven by ${}^{W}D_{t}^{\alpha,\beta}$, and systematic numerical investigations based on convolution quadrature and related time discretization schemes.

\section*{Acknowledgements}
The author thanks the Pudim AI Project for communicating the public staged
proof APP-0019 concerning the Bernstein range of the $W$-symbol. This
communication helped clarify the distinction between the Bernstein property,
the complete Bernstein property, and the subordination framework, and
contributed to improving the corrected version of the manuscript.

\end{document}